\documentclass[11pt,a4paper,reqno]{amsart}
\usepackage{color}
\usepackage{amsmath}
\usepackage{bm}
\usepackage{amsthm}
\usepackage{amsfonts}
\usepackage{amssymb}
\usepackage{enumerate}

 \newcommand{\oO}{\mathcal O}
 \newcommand{\p}{\varphi}

\newcommand\R{{\mathbf{R}}}

\newcommand\V{\mathcal{V}}

\newcommand\M{{\mathcal{M}}}

\newcommand\Z{{\mathbf{Z}}}

\newcommand\Q{{\mathbf{Q}}}

\newcommand\G{\mathbf{G}}
\newcommand\GN{\mathbf{G}_{N}}

\newcommand\T{\mathbf{T}}

\newcommand\bxi{\bm{\xi}}
\newcommand\bzero{\mathbf{0}}

\theoremstyle{plain}
  \newtheorem{theorem}{Theorem}

  \newtheorem{proposition}{Proposition}

\theoremstyle{remark}
  \newtheorem{remark}[subsection]{Remark}

\theoremstyle{definition}

\begin{document}

\title{Strongly badly approximable matrices in fields of power series}
\author{Th\'ai Ho\`ang L\^e and Jeffrey D. Vaaler}

\thanks{The research of the second author was supported by NSA grant, H98230-12-1-0254.}

\address{Department of Mathematics
The University of Texas at Austin,
1 University Station, C1200,
Austin, Texas 78712}
\email{leth@math.utexas.edu}
\email{vaaler@math.utexas.edu}

\begin{abstract}
We study the notion of strongly badly approximable matrices in the field of power series over a field $K$. We prove a transference principle in this setting, 
and show that such matrices exist when $K$ is infinite.
\end{abstract}

\maketitle

\section{Introduction}Let $K$ be a field and let $L = K((x^{-1}))$ be the field of formal Laurent series with coefficients in $K$.  That is, 
the nonzero elements of $L$ consist of all formal Laurent series
\begin{equation}\label{comp51}
F(x) = \sum_{m = -\infty}^M a(m) x^m,
\end{equation}
where each coefficient $a(m)$ is in $K$, $a(M) \not= 0_K$, and $M$ is an arbitrary integer.  We write $0_L$ for the
identically zero Laurent series, which is also in $L$.  Addition and multiplication in $L$ are defined in the obvious way.  
Then we define $|\ |:L \rightarrow [0, \infty)$ by $|0_L| = 0$, and
\begin{equation}\label{comp52}
|F| = e^M,
\end{equation}
where $F(x) \not= 0_L$ in $L$ is given by (\ref{comp51}).  It follows that $|\ |:L \rightarrow [0, \infty)$ is a 
discrete, non-archimedean absolute value on $L$, and the resulting metric space $\bigl(L, |\ |\bigr)$ is complete.  In 
this situation we define the subring of $L$-adic integers 
\begin{equation*}\label{comp53}
\oO_L = \big\{F \in L : |F| \le 1\big\}, 
\end{equation*}
and its unique maximal ideal
\begin{equation*}\label{comp54}
\M_L = \big\{F \in \oO_L : |F| < 1\big\}.
\end{equation*}
Clearly the residue class field $\oO_L/\M_L$ is isomorphic to $K$.  We recall that $\oO_L$ is compact if and only if the 
residue class field $\oO_L/\M_L$ is finite (see \cite[Chapter 4, Corollary to Lemma 1.5]{cassels1986}.)  Obviously 
$\M_L \subseteq \oO_L$ is the principal ideal generated by the element $x^{-1}$.

It is clear that the polynomial ring $K[x]$ can be embedded in $L$ by simply regarding a polynomial
\begin{equation*}\label{comp55}
P(x) = \sum_{m = 0}^M \xi(m) x^m
\end{equation*}
as a Laurent series with $\xi(m) = 0_K$ for integers $m \le -1$.  In what follows we will always identify $K[x]$ with
its image in $L$.  Let $\p: L \rightarrow \M_L$ be the map defined by $\p(0_L) = 0_L$, and defined on nonzero elements $F$ in $L$ by
\begin{equation*}\label{comp56}
\p(F) = \p\biggl( \sum_{m = -\infty}^M a(m) x^m\biggr) =  \sum_{m = -\infty}^{-1} a(m) x^m.
\end{equation*}
It is easy to check that $\p$ is a surjective homomorphism of additive groups, and
\begin{equation}\label{comp57}
\ker\{\p\} = K[x].
\end{equation}
We put $\T = L/K[x]$. Then $\p$ induces an isomorphism of additive groups
\begin{equation}\label{comp58}
\overline{\p} : \T \rightarrow \M_L.
\end{equation}

The subset $K[x] \subseteq L$ is clearly discrete with respect to the metric topology induced by the absolute value $|\ |$.  In
particular, if $P(x) \not= Q(x)$ are polynomials in $K[x]$, then we have
\begin{equation}\label{comp59} 
1 \le |P - Q|.
\end{equation}
Next we define $\|\ \|:\T \rightarrow [0, 1)$ on the additive group $\T$ by
\begin{equation*}\label{comp60}
\|F\| = \min\big\{|F - P|: P \in K[x]\big\}.
\end{equation*}
Alternatively, $\|F\|$ is the distance in $L$ from $F$ to the nearest polynomial.
It is trivial to check that $\|\ \|:\T \rightarrow [0, 1)$ is a norm on $\T$ in the sense of Kaplansky \cite[Appendix 1]{kaplansky1977}.
Therefore the map
\begin{equation*}\label{comp61}
(F, G) \mapsto \|F - G\|
\end{equation*}
defines a metric in $\T$, and so induces a metric topology in this quotient group.

There is a natural $K[x]$-module structure on the additive group $\T$.  If $\bigl(P(x), F(x) + K[x]\bigr)$ is an element 
of $K[x] \times \T$ we define the product
\begin{equation}\label{comp70}
P(x)\bigl(F(x) + K[x]\bigr) = P(x)F(x) + K[x].
\end{equation}
Again it is easy to check that (\ref{comp70}) does give the additive group $\T$ the structure of a $K[x]$-module. 

We can formulate Diophantine approximation problems in this setting, with $K[x], K(x), L$ and $\T$ playing analogous roles to $\Z, \Q, \R$ and the unit interval $[0,1)$, respectively. 
Indeed, Davenport and Lewis \cite{dl} studied the analog of Littlewood's conjecture in 
$L$ and proved that this analog is actually \textit{false} when $K$ is infinite. Explicit counterexamples were later given by Baker \cite{baker}.
In the case where $K$ is finite, the analog of the Littlewood conjecture in $L$ is believed to be true, but still remains an open problem (see \cite{ab}). 

In \cite{lv}, motivated by inequalties regarding fractional parts of linear forms, we introduced the notion of strongly badly approximable matrices, 
which is a strengthening of the usual notion of badly approximable matrices.
Let $A=(\alpha_{mn})$ be an $M \times N$ matrix with entries in $\R/\Z$. Recall that $A$ is \textit{badly approximable} 
if there exists a constant $\beta(A)>0$ such that we have
\begin{equation}
 \beta(A) \leq \left( \max_{1\leq m \leq M} \left\| \sum_{n=1}^N \alpha_{mn} \xi_n \right\| \right)^{M} \left( \max_{1 \leq n \leq N} |\xi_n| \right)^N
\end{equation}
for all $\bxi \in \Z^N \setminus \{ \bzero \}$. Here $\| \alpha \|$ denotes the distance from $\alpha$ to the nearest integer.
We say that $A$ is \textit{strongly badly approximable} if there exists a constant $\gamma(A)>0$ such that we have
\begin{equation}
 \gamma(A) \leq \left( \prod_{m=1}^{M} \left\| \sum_{n=1}^N \alpha_{mn} \xi_n \right\| \right) \left( \prod_{n=1}^{N} (|\xi_n|+1) \right)
\end{equation}
for all $\bxi \in \Z^N \setminus \{ \bzero \}$. Clearly if $A$ is strongly badly approximable, then it is badly approximable. If $M=N=1, A=(\alpha)$, then $A$ is strongly badly approximable if 
$\alpha$ is a badly approximable number. If $M=2, N=1, A=( \alpha \, \beta)$ then $A$ is strongly badly approximable if and only if there is a constant $\gamma>0$ such that
\[
 \gamma \leq \| k \alpha\| \| k \beta \| (|k|+1)
\]
for all $k \in \Z \setminus \{ 0\}$. That is, $(\alpha, \beta)$ is a counterexample to Littlewood's conjecture. In \cite{lv}, using the geometry of numbers, 
we obtained a \textit{transference principle} for strongly badly approximable matrices, namely that $A$ is strongly badly approximable if and only 
if its transpose $A^t$ is. 

In this note, we consider the notion of strongly badly approximable matrices in $L$. 
We prove an explicit criterion for strongly badly approximable matrices in this setting, from which the transference principle readily follows. 
We also show that, when $K$ is infinite, then strongly badly approximable matrices \textit{do exist}.

Consider an $M \times N$ matrix $A=(\alpha_{mn})$ with entries in $\T$. We say that $A$ is \textit{strongly badly approximable} 
if there exists a constant $C_1=C_1(A)>0$ such that 
\begin{equation} \label{eq:def}
 e^{-C_1} < \left( \prod_{m=1}^{M} \left\| \sum_{n=1}^N \alpha_{mn} \xi_n \right\| \right) \left( \prod_{n=1}^{N} \max(|\xi_n|,1) \right)
\end{equation}
for all $\bxi=(\xi_1, \ldots, \xi_N) \in K[x]^N \setminus \{ \bzero \}$. Clearly, $C_1$ is necessarily greater than $M$.

Let us denote the set of all elements $\xi \in K(x)$ with $|\xi| < e^{N}$ by $\GN$. Given nonnegative integers $U_1, \ldots, U_N$, and
$V_1, \ldots, V_M$, let us consider the set of all $\bxi=(\xi_1, \ldots, \xi_N) \in K[x]^N$ satisfying the conditions
\begin{align}
 |\xi_n| & < e^{U_n} \, \textrm{ for any } n=1, \ldots, N, \label{eq:cond1}\\
\left\| \sum_{n=1}^N \alpha_{mn} \xi_n \right\| & < e^{-V_m} \, \textrm{for any } m=1, \ldots, M.  \label{eq:cond2}
\end{align}
Clearly, the solutions to the systems (\ref{eq:cond1}) and (\ref{eq:cond2}) form a subvector space of 
\begin{equation*}
\G_{U_1} \times \cdots \times \G_{U_N} \subset K[x]^N,
\end{equation*}
which we denote by 
$\V(U_1, \ldots, U_N; V_1, \ldots, V_M)$. Furthermore, the condition (\ref{eq:cond2}) determines $V_1 + \cdots + V_M$ equations (whose unknowns are the coefficients of $\xi_1, \ldots, \xi_N$). Thus we immediately have the inequality
\begin{equation} \label{eq:dim1}
 \dim \V(U_1, \ldots, U_N; V_1, \ldots, V_M) \geq \max(0, U_1+\cdots + U_N - V_1 -\cdots - V_M).
\end{equation}
One may regard this as a generalization of Dirichlet's theorem.

\section{A first characterization}
We now prove our first characterization of strongly badly approximable matrices in $\T$.

\begin{proposition} \label{prop:1}
Let $A = (\alpha_{mn})$ be an $M \times N$ matrix with entries in $\T$, where $m = 1, 2, \dots , M$ indexes rows and $n = 1, 2, \dots , N$
indexes columns.  Then $A$ is strongly badly approximable if and only if 
there exists a constant $C_2 \geq 0$ such that for any nonnegative integers $U_1, \ldots, U_N$ and $V_1 \ldots, V_M$, we have
\begin{equation} \label{eq:dim2}
 \dim \V(U_1, \ldots, U_N; V_1, \ldots, V_M) \leq \max(0, U_1+\cdots + U_N - V_1 -\cdots - V_M) + C_2.
\end{equation}
\end{proposition}
In other words, $A$ is strongly badly approximable if and only if the inequality (\ref{eq:dim1}) is essentially best possible. One could regard Proposition \ref{prop:1} as an analog of \cite[Lemma 4.2]{lv}.

\begin{proof}
Suppose $A$ is strongly badly approximable, that is, there is a constant $C_1$ such that (\ref{eq:def}) holds for any $\bxi \in K[x]^N \setminus \{ \bzero \}$.
We first claim that 
\begin{equation*}
\V(W_1, \ldots, W_N; V_1, \ldots, V_M) = \{\bzero\}
\end{equation*}
whenever 
\begin{equation*}
W_1 + \cdots + W_N \leq V_1 + \cdots + V_M + M - C_1.
\end{equation*}

Indeed, suppose there is $\bxi \in \V(W_1, \ldots, W_N; V_1, \ldots, V_M), \bxi \neq \bzero$. By definition of \\
$\V(W_1, \ldots, W_N; V_1, \ldots, V_M)$, we have
\begin{align}
 \max \left( |\xi_n| ,1 \right) & \leq e^{W_n} \, \textrm{ for any } n=1, \ldots, N \\
\left\| \sum_{n=1}^N \alpha_{mn} \xi_n \right\| & \leq e^{-V_m -1} \, \textrm{for any } m=1, \ldots, M.
\end{align}
Therefore, 
\begin{equation} \label{eq:prod}
 \left( \prod_{m=1}^{M} \left\| \sum_{n=1}^N \alpha_{mn} \xi_n \right\| \right) \left( \prod_{n=1}^{N} \max(|\xi_n|,1) \right) \leq e^{W_1 + \cdots + W_N-V_1 -\cdots - V_M -M}
\end{equation}
This contradicts (\ref{eq:def}) if $W_1 + \cdots + W_N \leq V_1 + \cdots + V_M + M -C_1 $.

We now show that (\ref{eq:dim2}) holds for 
\begin{equation} \label{eq:c1}
C_2:=C_1-M 
\end{equation}
(note that $C_2 >0$). If $U_1 + \cdots + U_N \leq V_1 + \cdots + V_M - C_2$, then (\ref{eq:dim2}) is automatically true by what we just proved. Suppose
\[
U_1 + \cdots + U_N \geq V_1 + \cdots + V_M - C_2.
\]
We can find integers $0 \leq W_n \leq U_n$ such that 
\begin{equation*}
W_1 + \cdots + W_N =  \max( V_1 + \cdots + V_M - C_2, 0).
\end{equation*}
Then $\G_{W_1} \times \cdots \times \G_{W_N}$ and $\V(U_1, \ldots, U_N; V_1, \ldots, V_M)$ are two subvector spaces of $\G_{U_1} \times \cdots \times \G_{U_N}$, whose intersection is $\V(W_1, \ldots, W_N; V_1, \ldots, V_M) = \{\bzero \}$. It follows that
\begin{align*}
\dim \V(U_1, \ldots, U_N; V_1, \ldots, V_M) & \leq U_1 + \cdots + U_N - (W_1 + \cdots + W_N) \\
 & = U_1 + \cdots + U_N - \max( V_1 + \cdots + V_M -C_2, 0)\\
 & \leq U_1+\cdots + U_N - V_1 -\cdots - V_M + C_2\\
 & \leq \max(0, U_1+\cdots + U_N - V_1 -\cdots - V_M) + C_2
\end{align*}
as desired.

For the reverse direction, let us assume that (\ref{eq:dim2}) holds for some constant $C_2$. Let $\bxi$ be an arbitrary element in $K[x]^N \setminus \{ \bzero \}$. 
Our goal is to find a lower bound for
\[
P =\left( \prod_{m=1}^{M} \left\| \sum_{n=1}^N \alpha_{mn} \xi_n \right\| \right) \left( \prod_{n=1}^{N} \max(|\xi_n|,1) \right).
\]
Put $\max(|\xi_n|,1) = e^{U_n}$ for $n=1, \ldots, N$. First we observe that $\left\| \sum_{n=1}^N \alpha_{mn} \xi_n \right\| \neq 0$ for any $m=1, \ldots, M$. 
Indeed, suppose for a contradiction that $\left\| \sum_{n=1}^N \alpha_{1n} \xi_n \right\| = 0$. 
Let $W$ be any integer greater than $C_2$. Then for any $V \geq 0$, the vector space $\V(U_1 + W+1, \ldots, U_N+W+1; V, 0, \ldots, 0)$ has dimension at least $W$ (since it contains $f \bxi$, for any $f \in \G_{W}$). On the other hand, in view of (\ref{eq:dim2}), its dimension cannot exceed $C_2$ if $V$ is sufficiently large.

Let us put $\left\| \sum_{n=1}^N \alpha_{mn} \xi_n \right\|=e^{-V_m}$, where $V_m$ are nonnegative integers. Then we want to find a lower bound for
\begin{equation}\label{jeff1}
\sum_{n=1}^{N} U_n - \sum_{m=1}^{M} V_m . 
\end{equation}
We can assume that (\ref{jeff1}) is negative, otherwise, we are already done.  Let $W$ be the smallest integer such that
\begin{equation} \label{eq:w}
\sum_{n=1}^{N} (W + U_n) - \sum_{m=1}^{M} \max( 0, V_m - W) \geq 0.
\end{equation}
Clearly $W$ exists, and $W \geq 1$. From (\ref{eq:w}) we immediately have
\begin{equation} \label{eq:w2}
\sum_{n=1}^{N} U_n - \sum_{m=1}^{M} V_m \geq - (M+N)W.
\end{equation}
By minimality of $W$, we have
\[
\sum_{n=1}^{N} (W-1 + U_n) - \sum_{m=1}^{M} \max( 0, V_m - W +1) < 0.
\]
Upon observing that $\max( 0, V_m - W +1) \leq \max( 0, V_m - W) + 1$, we have 
\[
\sum_{n=1}^{N} (W + U_n) - \sum_{m=1}^{M} \max( 0, V_m - W) < M+N.
\]
Let us now consider the space 
\begin{equation*}
\V (W+U_1, \ldots, W+ U_N; \max(0, V_1-W), \ldots, \max(0, V_M-W)).
\end{equation*}
On the one hand, by hypothesis, its dimension is at most
\[
\max \left( 0, \sum_{n=1}^{N} (W + U_n) - \sum_{m=1}^{M} \max( 0, V_m - W) \right)+ C_2 < (M+N) + C_2.
\]
On the other hand, its dimension is at least $W$, since it contains $f \bxi$ for any $f \in \G_{W}$. Therefore,
\begin{equation} \label{eq:w3}
W<(M+N) + C_2.
\end{equation}
Combining (\ref{eq:w2}) and (\ref{eq:w3}), we see that (\ref{eq:def}) holds for 
\begin{equation} \label{eq:c2}
C_1 : = (M+N)^2 + (M+N)C_2.
\end{equation}
\end{proof}
\begin{remark}One can contrast the values of $C_1$ and $C_2$ given by (\ref{eq:c1}) and (\ref{eq:c2}). It is an interesting problem to determine if these values are tight.
\end{remark}
\section{A matrix interpretation}
The pleasantness of working in $L$ is that we can write out any element of $L$ in terms of its coordinates 
and express Diophantine inequalities in terms of linear equations.
For any element $\alpha=\sum_{i=1}^{\infty} a_{i} x^{-i} \in \T$ and nonnegative integers $U, V$, let us denote by $\M_{U,V}(\alpha)$ the matrix
\begin{equation}
\M_{U;V}(\alpha)=
\begin{pmatrix}
a_1 	& a_2 & \cdots & a_{U}\\
a_2 	& a_3 & \cdots & a_{U+1} \\
\vdots  & \vdots & \ddots & \vdots  \\     
a_{V} 	& a_{V+1}& \cdots & a_{U+V-1} 
\end{pmatrix}.
\end{equation}
In particular, when $U=0$ or $V=0$, $\M_{U;V}(\alpha)$ is the empty matrix.

Given nonnegative integers $U_1, \ldots, U_N, V_1, \ldots, V_M$ and an $M \times N$ matrix $A=(\alpha_{mn})$ with entries in $\T$, 
the conditions (\ref{eq:cond1}) and (\ref{eq:cond2}) represent a system of $V_1 + \cdots + V_M$ linear equations in $U_1 + \cdots + U_N$ variables, 
which can be written out explicitly as follows. Suppose that
\begin{equation} \label{eq:alphamn}
 \alpha_{mn} = \sum_{i=1}^{\infty} a^{(mn)}_{i} x^{-i} \qquad (1 \leq m \leq M, 1 \leq n \leq N).
\end{equation}
Let us write
\begin{equation} \label{eq:xin}
 \xi_{n} = \sum_{j=0} ^{U_n -1} t^{(n)}_j x^j \qquad (1 \leq n \leq N)
\end{equation}
where we regard the $t^{(n)}_j$ as variables. Then the conditions (\ref{eq:cond1}),( \ref{eq:cond2}) amount to the system
\[
 \sum_{n=1}^N \sum_{j=0} ^{U_n -1} t^{(n)}_j a^{(mn)}_{i+j} = 0
\]
for any $i=1, \ldots V_m$ and $m=1, \ldots, M$. It is straightforward to see that the matrix of this system is
\begin{equation} \label{eq:m}
\M_{U_1, \ldots, U_N;V_1, \ldots, V_M}(A)=
\begin{pmatrix}
\M_{U_1;V_1}(\alpha_{11})	& \M_{U_2;V_1}(\alpha_{12}) & \cdots & \M_{U_N;V_1}(\alpha_{1N})\\
\M_{U_1;V_2}(\alpha_{21})	& \M_{U_2;V_1}(\alpha_{22}) & \cdots & \M_{U_N;V_1}(\alpha_{2N}) \\
\vdots  & \vdots & \ddots & \vdots  \\     
\M_{U_1;V_M}(\alpha_{M1})	& \M_{U_2;V_M}(\alpha_{M2}) & \cdots & \M_{U_N;V_M}(\alpha_{MN}) 
\end{pmatrix}.
\end{equation}

We thus arrive at another characterization of strongly badly approximable matrices.

\begin{proposition} \label{prop:2}
The $M \times N $ matrix $A$ is strongly badly approximable if and only if there is a constant $C_2 \geq 0$ such that for any nonnegative integers 
$U_1, \ldots, U_N, V_1, \ldots, V_M$, the matrix $\M_{U_1, \ldots, U_N;V_1, \ldots, V_M}(A)$ defined in {\rm (\ref{eq:m})} has rank at least
\begin{equation} \label{eq:c3}
\min \left( \sum_{n=1}^{N} U_n, \sum_{m=1}^{M} V_m \right) - C_2.
\end{equation}
\end{proposition}
\begin{proof}
This follows from Proposition \ref{prop:1}, the rank-nullity theorem, and the fact that
\begin{equation}
\sum_{n=1}^{N} U_n - \max \left(0, \sum_{n=1}^{N} U_n - \sum_{m=1}^{M} V_m \right) -C_2 =  \min \left(\sum_{n=1}^{N} U_n,  \sum_{m=1}^{M} V_m \right) -C_2. 
\end{equation}
\end{proof}
It is easy to see that the transpose of $\M_{U_1, \ldots, U_N;V_1, \ldots, V_M}(A)$ is $\M_{V_1, \ldots, V_M;U_1, \ldots, U_N}(A^t)$. Thus from Proposition \ref{prop:2}, we immediately have the following transference principle:

\begin{theorem} \label{th:transference}
A matrix $A$ with entries in $\T$ is strongly badly approximable if and only 
if its transpose $A^t$ is strongly badly approximable.
\end{theorem}
\begin{remark}
Propositions \ref{prop:1}, \ref{prop:2} show that, if $A$ is strongly badly approximable, we can choose $C_1(A^t)=(M+N)^2 + (M+N) C_1(A)$, where $C_1(A)$ is the constant defined in (\ref{eq:def}). One can compare this to the bound (8.30) in \cite{lv}.
\end{remark}

Given Proposition \ref{prop:2}, we will establish the existence of strongly badly approximable matrices when $K$ is infinite. This follows from the following stronger statement.
\begin{theorem} \label{th:existence}
Suppose $K$ is infinite. Then for any $M$ and $N$, there exists an $M \times N$ matrix $A=(\alpha_{mn})$ with entries in $\T$ 
satisfying the following property.
\begin{enumerate}
\item[(*)] 
For any nonnegative integers $U_1, \ldots, U_N$, $V_1, \ldots, V_M$ with $\sum_{n=1}^N U_n = \sum_{m=1}^M V_m$, the square matrix $\M_{U_1, \ldots, U_N;V_1, \ldots, V_M}(A)$ is non-singular.
\end{enumerate}
\end{theorem}
For $1 \times N$ matrices, this is a result of Bumby \cite{bumby}. Our argument is a generalization of his.
\begin{proof}
We prove Theorem \ref{th:existence} by induction on $M+N$. When $M=0$ or $N=0$, the theorem is vacuously true. 
Suppose $M, N \geq 0$ and we know already the existence of an $M \times N$ matrix $A'$ satisfying (*). 
We will show how to add one column to $A'$ such that the new $M \times (N+1)$ matrix retains this property. By symmetry, we can also add one row to $A'$. This way, 
we can create $M \times N$ matrices satisfying (*) for any $M$ and $N$. (Note that when $M=N=0$, then this process creates badly approximable elements in $L$.)

Suppose the $M \times N$ matrix $A'=(\alpha_{mn})$ 
satisfies (*). We will find $\alpha_{m, N+1} \in \T$ ($1 \leq m \leq M$) such that the $M \times (N+1)$ matrix $A=(\alpha_{mn})$
satisfies (*).

Suppose 
\begin{eqnarray*}
\alpha_{mn} = \sum_{i=1}^{\infty} a^{(mn)}_i x^{-i} \qquad (1 \leq m \leq M, 1 \leq n \leq N+1). 
\end{eqnarray*}
Our goal is to construct $M$ sequences 
\begin{equation*}
\left( a^{(m,N+1)}_i \right)_{i=1}^{\infty},\qquad   (1 \leq m \leq M), 
\end{equation*}
such that for integers $U_1, \ldots, U_N, U_{N+1}$ and $V_1, V_2, \ldots, V_M$, with 
\begin{equation*}
U_1+ \cdots + U_N + U_{N+1} = V_1 + \cdots + V_M,
\end{equation*}
the square matrix
\begin{align*}
\M_{U_1, \ldots, U_N, U_{N+1};V_1, V_2, \ldots, V_M}&(A)\\
&= \begin{pmatrix}
a^{(11)}_1 	& \cdots 	& a^{(11)}_{U_1}   	& \cdots   		& a^{(N+1,1)}_1   & \cdots 	& a^{(N+1,1)}_{U_{N+1}}\\
\vdots  	& \ddots 	& \vdots     	   	&   			& \vdots	   & \ddots	& \vdots               \\
a^{(11)}_{V_1}  & \cdots 	& a^{(11)}_{U_1+V_1-1} & \cdots			& a^{(N+1,1)}_{V_1}  & \cdots 	& a^{(N+1,1)}_{U_{N+1} + V_1 -1}	\\
\vdots 		& \vdots 	& \vdots 		& 	 		& \vdots	   & \vdots 	& \vdots \\ 
a^{(M1)}_1 	 & \cdots 	& a^{(M1)}_{U_1}       & \cdots			& a^{(N+1,M)}_1  & \cdots & a^{(N+1,M)}_{U_{N+1}}\\
\vdots  	& \ddots 	& \vdots        	& 	     		&\vdots	 	& \ddots 		& \vdots \\
a^{(M1)}_{V_M}  & \cdots 	& a^{(M1)}_{U_1+V_M-1} & \cdots 		& a^{(N+1,M)}_{V_M} & \cdots & a^{(N+1,M)}_{V_M+ U_{N+1}-1}
\end{pmatrix}
\end{align*}
is non-singular.

To this end, we will construct the $M$-tuples $\left( a^{(m,N+1)}_L \right)_{1 \leq m \leq M}$ indexed by $L$ recursively. Let us refer to to 
quantity $\max_{1 \leq m \leq M} (U_{N+1}+V_m-1)$ as the \textit{order} of the matrix 
\begin{equation*}
\M_{U_1, \ldots, U_N, U_{N+1};V_1, V_2, \ldots, V_M}(A).
\end{equation*}
Suppose all the tuples $\left( a^{(m,N+1)}_{\ell} \right)_{1 \leq m \leq M}$, with $1 \leq \ell \leq L-1$, are already determined in such a way that all matrices of order smaller than $L$ are non-singular. 
We want to find $\left( a^{(m,N+1)}_L \right)_{1 \leq m \leq M}$ such that all the matrices
$\M_{U_1, \ldots, U_N, U_{N+1};V_1, V_2, \ldots, V_M}(A)$ satisfying 
\begin{itemize}
\item $U_1 + \cdots + U_{N+1} = V_1 + \cdots + V_M$
\item $\max_{1 \leq m \leq M} (U_{N+1}+V_m-1) = L$
\end{itemize}
have non-zero determinant.

It is clear that the number of such matrices is finite. For any such matrix, by expanding along the last column, the non-zero determinant condition corresponds to an equation of the form
\begin{equation} \label{eq:det}
\sum_{m=1}^M r_{m} a^{(m,N+1)}_L + r_0 \neq 0 
\end{equation}
where $r_0, \ldots, r_M \in K$. For each $1 \leq m \leq M$, $r_m$ is either 0 or the determinant of a matrix of lower order (it is 0 if $a^{(m, N+1)}_L$ is not present in $\M_{U_1, \ldots, U_N, U_{N+1};V_1, V_2, \ldots, V_M}(A)$), but at least one of them is 
nonzero. The number of equations of type (\ref{eq:det}) is finite. Since $K$ is infinite, a choice of $\left( a^{(m,N+1)}_L \right)_{1 \leq m \leq M}$ is certainly possible.
\end{proof}

\section{Explicit examples of strongly badly approximable matrices}
In this section, we discuss about explicit examples of strongly badly approximable matrices. 
As mentioned earlier, Baker \cite{baker} gave explicit counterexamples to Littlewood's conjecture in $K$ in the case where $K$ has characteristic zero. 
He also pointed out that the same method can be used to show that any $1 \times N$ matrix of the form 
\[
 A = \begin{pmatrix}
e^{\lambda_1/x} 	& e^{\lambda_2/x} 	& \cdots &  e^{\lambda_N/x}
\end{pmatrix},
\]
where $\lambda_1, \ldots, \lambda_N$ are distinct elements of $K$, is (in our language) strongly badly approximable. Here for any $\lambda \in K$, 
$e^{\lambda/x}$ is the formal power series
\[
 e^{\lambda/x} = \sum_{i=0}^{\infty} \frac{\lambda^{i}}{i!} x^{-i}.
\]

In a different context, Jager \cite{jager} studied the notion of \textit{perfect systems} of power series (see the definition in \cite[p. 196]{jager}). 
This notion was developped by Mahler \cite{mahler} and inspired by Hermite's proof of the transcendence of $e$. 
By examining the underlying linear equations, it is easy to see that if the system 
$(f_1(x), \ldots, f_M(x))$ is perfect (where each $f_i$ is a power series of the form $f_i(x)=\sum_{j=0}^{\infty} a^{(i)}_j x^i$), then the matrix
\[
 A = \begin{pmatrix}
f_1\left(x^{-1} \right) \\
f_2\left(x^{-1} \right) \\
\vdots \\
f_M\left(x^{-1} \right)
\end{pmatrix}
\]
is strongly badly approximable.

Jager then gave some examples of perfect systems. If $\lambda_1, \ldots, \lambda_N$ are distinct elements of $K$, then the system 
$(e^{\lambda_1 x}, e^{\lambda_2 x}, \ldots, 	e^{\lambda_N x})$ is perfect. Coupled with the transference principle, 
this recovers Baker's result (with a simpler proof). He also showed that if $w_1, \ldots, w_N$ are elements of $K$, no two of which differ by an integer, 
then the system $\left( (1-x)^{w_1}, \ldots, (1-x)^{w_N} \right)$ is also perfect. Here $(1-x)^w$ is the formal power series
\[
 (1-x)^{w} = \sum_{i=0}^{\infty} (-1)^i \binom{w}{i} x^i.
\]
This gives another example of $1 \times N$ (hence $N \times 1$) strongly badly approximable matrices. However, it seems to us that 
neither Baker nor Jager's method can be extended to give explicit examples of $M \times N$ for \textit{some} $M, N >1$. It is therefore an interesting problem to 
give explicit examples of strongly badly approximable matrices of arbitrary dimensions.


\begin{thebibliography}{10}

\bibitem{ab}
B. Adamczewski, Y. Bugeaud, \emph{On the Littlewood conjecture in fields of power series}, Probability and Number Theory -- Kanazawa 2005, Adv. Stud. Pure Math. Vol. 49 (2007), 1--20.  
\bibitem{baker}
A. Baker, \emph{On an analogue of Littlewood's Diophantine approximation problem},
Michigan Math. J. 11 (1964), 247--250. 

\bibitem{bumby}
R. T. Bumby, \emph{On the analog of Littlewood's problem in power series fields},
Proceedings of the American Mathematical Society 18, No. 6 (1967), 1125--1127.

\bibitem{cassels1986}
J.~W.~S.~Cassels,
\newblock {\em Local Fields},
\newblock Cambridge U. Press, 1986.

\bibitem{dl}
H. Davenport, D. J. Lewis, \emph{An analogue of a problem of Littlewood},
Michigan Math. J. 10 (1963), 157--160. 

%\bibitem{kristensen}
%S. Kristensen, \emph{Badly approximable systems of linear forms over a field of formal series},
%J. Th\'{e}or. Nombres Bordeaux 18 (2006), no. 2, 421--444. 
\bibitem{jager}
H. Jager, \emph{A multidimensional generalization of the Pad\'e table. II},
Indag. Math. 26 (1964), 199--211.

\bibitem{kaplansky1977}
I.~Kaplansky,
\newblock{\em Set theory and metric spaces},
\newblock Chelsea, New York, 1977.

\bibitem{lv} 
T. H. L\^e, J. D. Vaaler, \emph{Sums of products of fractional parts}, submitted. Available at http://arxiv.org/abs/1309.1506.

\bibitem{mahler}
K. Mahler, \emph{Perfect systems}, Compositio Math. 19 (1968) 95--166.

%\bibitem{schmidt1}
%W. M. Schmidt, \emph{Badly approximable systems of linear forms}, J. Number Theory 1 (1969), 139--154. 

%\bibitem{schmidt2}
%W. M. Schmidt, \textbf{Diophantine Approximation}, Lecture Notes in Mathematics 785, Springer, Berlin (1980). 

\end{thebibliography}
\end{document}